\documentclass[12pt]{amsart}
\usepackage{amssymb}
\usepackage{color}
\usepackage[margin=2.9cm]{geometry}             
\usepackage{color,ulem}
\usepackage{graphicx}
\usepackage{amssymb, enumerate,bm}
\usepackage[matrix,arrow,ps]{xy}
\usepackage{epstopdf, amsmath}
\usepackage{paralist}
\begin{document}

\def\dbl{[\hskip -1pt[}
\def\dbr{]\hskip -1pt]}

\title
{New examples of  CR submanifolds for the  $2$-jet determination problem}
\author{
Florian Bertrand, Martin Kol\'a\v{r} and Francine Meylan
}
%\address{F. Meylan: Institut de Math\'ematiques, Universit\'e de Fribourg, 1700% Perolles, Fribourg, Switzerland}
%\email{francine.meylan@unifr.ch}
%\hfil\break
%\abstract \endabstract
%\keywords \endkeywords
%\subjclass{32H02}
%\thanks{\noindent 2000 {{\em Mathematics Subject Classification.} 32H02, 32H35  \\
%The  author was partially supported by Swiss NSF Grant
%2100-063464.00/1}}
%\date{\number\year-\number\month-\number\day}
%\enddate
%\loadeufm

\def\Label#1{\label{#1}}
\def\1#1{\ov{#1}}
\def\2#1{\widetilde{#1}}
\def\6#1{\mathcal{#1}}
\def\4#1{\mathbb{#1}}
\def\3#1{\widehat{#1}}
\def\K{{\4K}}
\def\LL{{\4L}}
\def\H{{\4H}}
\def\C{{\4C}}
\def\R{{\4R}}
\def \MM{{\4M}}

\def\Re{{\sf Re}\,}
\def\Im{{\sf Im}\,}

\numberwithin{equation}{section}
\def\s{s}
\def\k{\kappa}
\def\ov{\overline}
\def\span{\text{\rm span}}
\def\ad{\text{\rm ad }}
\def\tr{\text{\rm tr}}
\def\xo {{x_0}}
\def\Rk{\text{\rm Rk\,}}
\def\sg{\sigma}
\def \emxy{E_{(M,M')}(X,Y)}
\def \semxy{\scrE_{(M,M')}(X,Y)}
\def \jkxy {J^k(X,Y)}
\def \gkxy {G^k(X,Y)}
\def \exy {E(X,Y)}
\def \sexy{\scrE(X,Y)}
\def \hn {holomorphically nondegenerate}
\def\hyp{hypersurface}
\def\prt#1{{\partial \over\partial #1}}
\def\det{{\text{\rm det}}}
\def\wob{{w\over B(z)}}
\def\co{\chi_1}
\def\po{p_0}
\def\fb {\bar f}
\def\gb {\bar g}
\def\Fb {\ov F}
\def\Gb {\ov G}
\def\Hb {\ov H}
\def\zb {\bar z}
\def\wb {\bar w}
\def \qb {\bar Q}
\def \t {\tau}
\def\z{\chi}
\def\w{\tau}
\def\Z{\zeta}
\def\phi{\varphi}
\def\eps{\varepsilon}
\def\5#1{\mathfrak{#1}}

\newcommand{\transp}{\,^t}

\def \T {\theta}
\def \Th {\Theta}
\def \L {\Lambda}
\def\b {\beta}
\def\a {\alpha}
\def\o {\omega}
\def\l {\lambda}

\def \im{\text{\rm Im }}
\def \re{\text{\rm Re }}
\def \Char{\text{\rm Char }}
\def \supp{\text{\rm supp }}
\def \codim{\text{\rm codim }}
\def \Ht{\text{\rm ht }}
\def \Dt{\text{\rm dt }}
\def \hO{\widehat{\mathcal O}}
\def \cl{\text{\rm cl }}
\def \bR{\mathbb R}
\def \bS{\mathbb S}
\def \bK{\mathbb K}
\def \bD{\mathbb D}
\def \bC{\mathbb C}
\def \C{\mathbb C}
\def \N{\mathbb N}
\def \bL{\mathbb L}
\def \bZ{\mathbb Z}
\def \bN{\mathbb N}
\def \scrF{\mathcal F}
\def \scrK{\mathcal K}
\def \mc #1 {\mathcal {#1}}
\def \scrM{\mathcal M}
\def \cR{\mathcal R}
\def \scrJ{\mathcal J}
\def \scrA{\mathcal A}
\def \scrO{\mathcal O}
\def \scrV{\mathcal V}
\def \scrL{\mathcal L}
\def \scrE{\mathcal E}
\def \hol{\text{\rm hol}}
\def \aut{\text{\rm aut}}
\def \Aut{\text{\rm Aut}}
\def \J{\text{\rm Jac}}
\def\jet#1#2{J^{#1}_{#2}}
\def\gp#1{G^{#1}}
\def\gpo{\gp {2k_0}_0}
\def\emmp {\scrF(M,p;M',p')}
\def\rk{\text{\rm rk\,}}
\def\Orb{\text{\rm Orb\,}}
\def\Exp{\text{\rm Exp\,}}
\def\Span{\text{\rm span\,}}
\def\d{\partial}
\def\D{\3J}
\def\pr{{\rm pr}}

\def \CZZ {\C \dbl Z,\zeta \dbr}
\def \D{\text{\rm Der}\,}
\def \Rk{\text{\rm Rk}\,}
\def \CR{\text{\rm CR}}
\def \ima{\text{\rm im}\,}
\def \I {\mathcal I}

\def \M {\mathcal M}

\newtheorem{Thm}{Theorem}[section]
\newtheorem{Cor}[Thm]{Corollary}
\newtheorem{Pro}[Thm]{Proposition}
\newtheorem{Lem}[Thm]{Lemma}
\newtheorem{Conj}[Thm]{Conjecture}

\theoremstyle{definition}\newtheorem{Def}[Thm]{Definition}

\theoremstyle{remark}
\newtheorem{Rem}[Thm]{Remark}
\newtheorem{Exa}[Thm]{Example}
\newtheorem{Exs}[Thm]{Examples}

\def\bl{\begin{Lem}}
\def\el{\end{Lem}}
\def\bp{\begin{Pro}}
\def\ep{\end{Pro}}
\def\bt{\begin{Thm}}
\def\et{\end{Thm}}
\def\bc{\begin{Cor}}
\def\ec{\end{Cor}}
\def\bd{\begin{Def}}
\def\ed{\end{Def}}
\def\br{\begin{Rem}}
\def\er{\end{Rem}}
\def\be{\begin{Exa}}
\def\ee{\end{Exa}}
\def\bpf{\begin{proof}}
\def\epf{\end{proof}}
\def\ben{\begin{enumerate}}
\def\een{\end{enumerate}}
\newcommand{\zz}{(z,\bar z)}
 
\begin{abstract}
This paper addresses two questions related to mapping problems. In the first part of the paper,  we discuss  some  recent results  regarding the $2$-jet determination for  biholomorphisms between   smooth    (weakly) pseudoconvex  hypersurfaces; in  light of those, we formulate  a  general problem  and  illustrate it with examples.  In the second part of the paper, we 
provide a new  example of a real submanifold of codimension $7$ in $\C^{10}$ that is not strictly pseudoconvex and for which its germs of CR automorphisms are determined by their $2$-jets at a given point. 

\end{abstract}

\maketitle

\section*{Introduction}
Cartan's Uniqueness Theorem asserts that any   
 holomorphic map $F$  sending a bounded domain in $\Bbb C^N$ into itself, that fixes   some $p$  of the domain  with   $F'(p)=I,$  is  the identity map. ($F'(p)$ is the holomorphic derivative of $F$ at $p$  and $I$ is the identity operator).

\noindent Instead of considering holomorphic maps between bounded domains in $\Bbb C^N,$  the question  whether an analogous statement to Cartan's Theorem holds for germs  of  holomorphic maps in
$\Bbb C^{n+1}$ between boundaries of domains in $\Bbb C^{n+1},$ or more
generally, between generic submanifolds of codimension $d$  in $\Bbb C^{n+d},$  has attracted  a lot of research. See for instance the surveys \cite{Z02}, \cite{LM}, \cite{GKMS} and the references therein. 

\noindent We are interested   in   the   following two questions:  
\begin{enumerate}\label{preci}
\item Under what conditions (of regularity)  can we  compute   the optimal number  of derivatives needed to uniquely  determine   germs  of holomorphic maps, sending a germ at $p$   of  a   hypersurface in  $\Bbb C^{n+1}$  into itself fixing $p?$
\item Under what conditions  is  a  germ  of  a  CR automorphism  of  a sufficiently smooth  Levi nondegenerate     generic  real submanifold in  $\Bbb C^{n+d}$   
    uniquely determined by its $2$-jet at $p?$ 
\end{enumerate}

\noindent A fundamental result by Chern and Moser shows  that germs   of  biholomorphisms  sending  a  real-analytic  
hypersurface of type $2$ (in the sense of Kohn and Bloom-Graham)   into itself are uniquely determined by at most   $2$-jets if its Levi form is nondegenerate. See  \cite{CM}. Notice that $2$-jets
(and not less)  are required in the case of the Lewy hypersurface. 
 We emphasize that Question (2) is also relevant for the class of biholomorphisms, since 
 the above result by Chern and Moser   becomes false without any hypothesis on the Levi map (see e.g.  \cite{Me}, \cite{GM}).

\vspace{0.5cm}

\noindent  We start with Question (1).   The second and third authors, with Zaitsev,  have generalized in \cite{KMZ} this result to  any  smooth  hypersurface of type $ m>2,$ showing that at most  $(m-1)$-jets are needed to determine uniquely  automorphisms, provided that  its  associated  model hypersurface (the homogeneous hypersurface "given" by the minimal homogeneous  non pluriharmonic  polynomial  of degree $m$  with respect to the tangential variables  in the Taylor expansion of the smooth hypersurface)  is holomorphically nondegenerate. Notice that $(m-1)$-jets
(and not less)  are required for the hypersurface given by 
$
\{(z_1, z_2, w)\in  \Bbb C^{3} \ | \ 
 \Im w = \Re z_1 \bar z_2^{m-1}\}.
 $

Kim and the third author in \cite{KK} consider  the above   question  under the natural assumption of {\it pseudoconvexity}. They study    the case of smooth  (weakly) pseudoconvex  hypersurfaces under the condition that the Lie algebras of  infinitesimal CR automorphisms  of their (holomorphically nondegenerate) models have nonvanishing top degree components. See  Definition \ref{td}.  In particular, they show that only $2$-jets are needed to determine uniquely  their automorphisms.

 One of the goals  of this paper is to revisit \cite{KK} in order to explore the general case, that is,  when the top degree component of  the Lie algebra of  infinitesimal CR automorphisms  of   models of weakly pseudoconvex hypersurfaces  possibly  vanishes.   An example that illustrates  this situation is 
\be \label{pinson0} Let $M \subset  \mathbb C^{3}$ be   given by
 $$\Im w=|z_1|^8 + |z_1|^6 {\Re {z_1}}^2 + |z_2|^8 .$$ 
 The reader can check that $M$ is  a  holomorphically nondegenerate weakly pseudoconvex model with  vanishing of the top degree component  of its  symmetry  Lie algebra. See  Section \ref{bf}  for the  precise definitions and results.
\ee   
 In light of \cite{KK}, we may address the following problem:
\begin{itemize}
\item  Is   every automorphism of  a  pseudoconvex smooth  hypersurface  with holomorphically nondegenerate model
     uniquely determined by its $2$-jet? 
\end{itemize}

\vspace{0.5cm}

  We now focus on the  Question (2).
Let $M \subset \C^{n+d}$ be a  $\mathcal{C}^{4}$ generic real submanifold of real codimension $d\ge 1$ through $p=0$ given locally by the following system of equations 
\begin{equation*}
	\begin{cases}
		\Re   w_1= \transp\bar z A_1 z+ O(3)\\
		\ \ \ \ \vdots \\
		\Re   w_d = \transp\bar z A_d z+ O(3)
	\end{cases}
\end{equation*}
where $A_1,\hdots,A_d$ are Hermitian matrices of size $n$, and in the remainder O(3), the variables $z$ and $\Im w$  are respectively of weight one and two. 
We assume that $M$ is a
strongly Levi nondegenerate real submanifold, that is, such that  $\sum_{j=1}^d b_jA_j$  is invertible for some $b\in \R^d$.
For  $a \in \C^d$ sufficiently small,  we 
consider the  following  matrix equation 
\begin{equation*}
	PX^2 + AX +   \transp{\overline {P}}=0,
\end{equation*} 
 where $P:=\sum_{j=1}^d {{a}_j} A_j $ and   $A:=\sum_{j=1}^d (b_j -a_j-\overline{a_j})A_j,$
and we let   $X$ be the unique $n\times n$ matrix solution 
of this equation with $\|X\|<1.$ The matrix  $X$ plays a crucial role  in the study of stationary discs, a technique used for  the CR automorphism  $2$-jet determination  problem (see e.g. \cite{be-bl-me}, \cite{be-me},  \cite{tu3}, \cite{be-me2}).
In particular, it leads to an important factorization result that   generalizes a classical factorization theorem due to Lempert in \cite{le} for positive definite matrix $A$ (see Lemma 2.3 in \cite{be-me2}).

In \cite{be-me3}, we introduced the notion of  {\it  $\mathfrak{D}(a)$-nondegeneracy} (see definition \ref{definondeg100}) that generalizes the notion of {\it  $\mathfrak{D}$-nondegeneracy},   a sufficient condition to obtain  the $2$-jet determination of CR automorphisms (see  \cite{be-me},\cite{be-me2}).
If $M$ is strictly pseudoconvex, that is, when the matrix $A$ defined above can be chosen  to be positive definite,   Tumanov showed in \cite{tu3}  that it is automatically   $\mathfrak{D}(a)$-nondegenerate at $0$ for some small $a \in \R^d.$  He then showed, using this condition, that   any  germ  of  a  CR automorphism  of class $C^2$   is  
uniquely determined by its $2$-jet at $p.$ We also note that if  $M$ is a Levi nondegenerate hypersurface, then $M$  is  $\mathfrak{D}$-nondegenerate at $0$;  the first author  and  Blanc-Centi proved in  this case  that   any  germ  of  a  CR automorphism of class $C^3$   is  
uniquely determined by its $2$-jet at $p$   (see \cite{be-bl}).

We formulated in \cite{be-me3} the following conjecture:
\begin{itemize}
\item  If $M$ is   $\mathfrak{D}(a)$-nondegenerate at $0$ for some small $a \in \C^d,$ then any  germ  of CR automorphism  of class $C^3$   is 
uniquely determined by its $2$-jet at $0.$
\end{itemize}
 In this paper, we present an example of a model quadric  which is $\mathfrak{D}(a)$-nondegenerate at $0$ for some small $a \ne 0,$  which is not strictly pseudoconvex nor $\mathfrak{D}$-nondegenerate at $0,$  and which  satisfies the above conjecture.
	
\vspace{0.5cm}

The paper is organized as follows. Section 1 deals with the first question. In subsections \ref{bf1} and \ref{bf2}, we  give new proofs  of the results of \cite{KK}, starting with the   important class   of  (weakly) pseudoconvex hypersurfaces whose  models have  Lie  algebras of symmetries with nonvanishing top degree components: namely, the hypersurfaces whose models are given by " sums of squares". In subsection \ref{bf3}, we discuss the general case in $\mathbb C^{3},$  and obtain a positive answer to the above question for  (weakly) pseudoconvex  hypersurfaces with  "single chain" models (Theorem \ref{capi}).
 The general case in $\mathbb C^{3} $ will be  treated in a forthcoming paper. The second question is addressed in Section 2 where we essentially present an example of a quadric in $\C^{10}$ which is not strictly pseudoconvex nor 
 $\mathfrak{D}$-nondegenerate, and is $\mathfrak{D}(a)$-nondegenerate at $0$ for some small $a$, and such that its germs of CR automorphisms are 
uniquely determined by their $2$-jets at $0$.

\section{Part I: Results on the $2$-jet determination under the assumption of pseudoconvexity.}

\subsection{Preliminaries}\label{bf}
Let  $M \subset \mathbb C^{n+1}$ be a  smooth hypersurface,the light
and $p \in M $ be a  point of finite type $m$ in the sense of Bloom-Graham \cite{BG}.
We  consider
local holomorphic coordinates $(z,w)$ vanishing at $p$,
where $z =(z_1, z_2, ..., z_n)$ and  $z_j = x_j + iy_j$,
$w=u+iv$. The hyperplane $\{ v=0 \}$ is assumed to be tangent to
$M$ at $p$, hence  $M$  is described near $p$ as the graph of a uniquely
determined real valued function
\begin{equation*} v = \psi(z_1,\dots, z_n,  \bar z_1,\dots,\bar z_n,  u), \ d\psi(0) = 0.
\label{vp1}
\end{equation*}
{
  We may assume (\cite{BG}, \cite{K})
that
\begin{equation}\label{fifi}
\psi(z_1,\dots, z_n,  \bar z_1,\dots,\bar z_n,  u)=Q(z, \bar z) +o(|u|+|z|^m),
\end{equation}
where $Q(z, \bar z)$ is a nonzero homogeneous polynomial of degree $m$ with {no pluriharmonic} terms. Note that removing the pluriharmonic terms requires a change of coordinates }{in the variable $w$, which "absorbs" the pluriharmonic terms into $w$.}

 Reflecting on the  choice
of  variables in  {\eqref{fifi}}, the
variables $w$, $u$ and $v$ are given weight $1,$  $\dfrac {\partial }{\partial {w}}$ weight $-1,$ while 
the complex tangential variables $(z_1, \dots, z_n)$  are given weight $ \dfrac{1}{m}$ and  $\dfrac {\partial }{\partial {z_{j}}}, \ j=1, \dots n, $ weight $ -\dfrac{1}{m}.$ 
For this choice of weights, 
 {$Q(z, \bar z)$} is  a  weighted homogeneous polynomial of weighted  degree one, while $o(|u|+|z|^m)$ contains weighted homogeneous polynomials  of weighted degree  strictly bigger than one.

 Recall that the   derivatives of  weighted order $\kappa$  are  those of the form
 $$\dfrac{\partial^{\vert \alpha \vert + \vert \hat \alpha \vert + l}}{\partial z^{\alpha}\partial \bar z^{\hat \alpha} \partial u^l}, \ \ \ 
 \kappa=l +\dfrac{1}{m}\sum_{j=1}^n (\alpha_j + \hat \alpha_j).$$\emph{}
We may then rewrite  equation \eqref{fifi} as 
\begin{equation}\label{vp12}
 v = \psi(z,  \bar z,  u)= Q\zz + o(1),
\end{equation}
where $o(1)$ denotes a smooth function  whose derivatives  at $0$ of weighted order less than or equal to
one vanish. This motivates the following definition 
\bd\cite{FM} Let $M$ be given by  (\ref{vp12}).
We define  its {\it associated
 model hypersurface} 
by
\begin{equation*} M_0: = \{(z,w) \in \mathbb C^{n+1}\ | \
 v  = Q \zz \}. \label{2}\end{equation*}
We say that $M$ is a smooth perturbation of $M_0.$
\ed

We denote  by $\Aut (M,p),$  the stability group of $M,$ that is, the germs at $p$
of biholomorphisms mapping $M$ into itself and fixing $p,$
and  by $\hol (M,p),$ the set of  germs of holomorphic vector fields in
$\mathbb C^{n+1}$ whose real part 
is tangent to $M$. An element of $ \hol(M,p)$ is called a  real-analytic infinitesimal CR automorphism at $0$, or  a symmetry.

We  write 
$$ \hol(M_0,0) = \bigoplus_{ \mu \ge -1 }  \5g_{\mu},$$
where $\5g_{\mu}$ consists of weighted homogeneous vector fields
of weight $\mu.$ 
 We recall the following definition

\bd{A real-analytic hypersurface $M \subset \mathbb C^{n+1}$  is  {\it holomorphically nondegenerate
at $p \in M$} if there is no germ at $p$ of a holomorphic vector field $X$ tangent to $M.$}
 \ed
It  is shown in \cite{KMZ}  that, if  $M_0$ is holomorphically nondegenerate,   the Lie algebra of  real-analytic infinitesimal CR automorphisms $\hol(M_0, 0)$
 has   the following weighted grading:
\begin{equation}\label{grin}
\hol(M_0, 0) = \5g_{-1}\oplus \5g_{-1/m}  \oplus \5g_{0}
\oplus \bigoplus_{\tau=1}^{m-2}\5g_{\tau/m} \oplus \5g_{1-1/m} \oplus \5g_{1},
\end{equation}

with the following explicit description of the
graded components:
\begin{enumerate}
\item $\5g_{-1}=\{ a\d_{w} \ | \ a\in \R\}$, 

\item
$\5g_{-1/m} = \{\sum_{j} a^{j} \d_{z_{j}} +
g(z)\d_{w} \ | \ a^{j}\in \C,\, g \in\C_{m-1}[z], \,
2i\sum(a^{j}Q_{z_{j}} + \bar a^{j} Q_{\bar z_{j}}) = g-\bar g
 \}$, 

\item $\5g_{0} =\{\sum_{j}
f^{j}(z)\d_{z_{j}} + a w\d_{w} \ | \ f^{j} \in\C_{1}[z], \,a\in \R, \, 
\sum(f^{j}Q_{z_{j}} + \bar f^{j} Q_{\bar z_{j}}) = aQ
\}$,

\item  $\5g_c:= \bigoplus_{\tau=1}^{m-2}  \5g_{\tau/m}
= \bigoplus_{\tau=1}^{m-2} \{\sum_{j} f^{j}(z)\d_{z_{j}} \ | \ f^{j}\in \C_{\tau+1}[z],\,
\sum(f^{j}Q_{z_{j}} + \bar f^{j} Q_{\bar z_{j}}) = 0
\}$,

\item $\5g_{1-1/m}
= \{\sum_{j} (f^{j}(z) + a^{j}w)  \d_{z_{j}} + g(z)w \d_{w} \ | \
a^{j}\in \C, \, f^{j}\in \C_{m}[z], \, g\in\C_{m-1}[z], \,
\sum_{j} a^{j}\d_{z_{j}} +
g(z)\d_{w}\in\5g_{-1/m},\,
2\sum(f^{j}Q_{z_{j}} + \bar f^{j} Q_{\bar z_{j}} + 2iQ(a^{j}Q_{z_{j}} - \bar a^{j} Q_{\bar z_{j}}))
=Q(g+\bar g))
 \}$.
% , where  and $\sum_{j} f^{j} g_{z_{j}} = g^{2}$.

\item $\5g_{1}=\{
\sum_{j}f^{j}(z)w\d_{z_{j}}  + a w^{2}\d_{w} \ | \
f^{j}\in\C_{1}[z], \, a\in\R, \, \sum_{j} f^{j}(z) Q_{z_{j}} =
a Q \}$.

\end{enumerate}
 We recall that  elements in $\5g_{0}$ with $a=0$ are called {\it rigid}. We now introduce the following definition 
\bd \label{td} Under the assumption that $M_0$ is holomorphically nondegenerate,  $\5g_{1}$ is called the {\it top degree component of the Lie algebra $\hol(M_0, 0).$}
 \ed
  Zaitsev  asked whether
 the above  structure results may  be improved under the
{\it pseudoconvexity} assumption.
 
\noindent We remark that  $\5g_c=\{0\}$  in the Levi nondegenerate case.
 In  $\mathbb C^2,$  we have  $\5g_c=\{0\} $ for every model (see e.g. \cite{ELZ1}, \cite{FM11}).

\noindent Notice also that   in $\mathbb C^2,$
there are  three
different cases, two exceptional ones and one generic. The two exceptional ones
are the circular model $\Im w = |z|^
{2k}$ with a $4-$dimensional symmetry  Lie algebra,
and the tubular model $\Im w = (\Re z)^k$, which has a three dimensional
symmetry  Lie algebra. All other models fall into the generic case, which admits
a two dimensional  Lie algebra of  infinitesimal symmetries. 
  
The above  description of the Lie algebra of a holomorphically nondegenerate model  hypersurface  leads, in particular,  to the following theorem
\bt \label{cris}   Let  $M$ be  a smooth perturbation of $M_0$ given by  \eqref{2}. Suppose that $M_0$ is holomorphically nondegenerate and   $\5g_c=0.$
Then  every  element of  $\Aut (M,0)$ is uniquely determined by its $2$-jet.
\et

\subsection{The case of a sum of squares polynomial  model}\label{bf1}
We start with the case of a polynomial model given by a sum of squares, that is given by \eqref{2}, with
$$Q = \sum_{l=1}^k P_l \overline {P_l},$$  where $P_l,  \ l=1, \dots, k, $ are  homogeneous holomorphic polynomials of the same degree. Note that this model is pseudoconvex. We will need  the following definition 
\bd
 A linear vector field   in the variables $z$ given in Jordan normal form  is {\it real diagonal}  (resp. {\it imaginary diagonal}) if it  is diagonal with all the entries real (resp. imaginary). Without loss of generality, we say that $X$ is {\it real diagonal} (resp. {\it imaginary diagonal})   if its  Jordan normal form is real (resp. imaginary) diagonal.
\ed

\noindent Let  $X \in \5g_{0}$   be rigid,  given in  normal  Jordan form after a possible linear change of coordinates in the variables $z.$ We have  $X= X^{Re}+X^{Im} + X^{Nil},$  where  $X^{Re}$ (resp. $X^{Im}$) is real (resp. imaginary)  diagonal, and  $X^{Nil}$ is nilpotent.  Recall that if $X \in \5g_{0},$ then $X^{Re} \in \5g_{0},$ $X^{Im}\in \5g_{0},$  and $X^{Nil}\in \5g_{0}$ (see e.g.  \cite{KMZ}, \cite{KM22}).

\noindent The following theorem  deals with symmetries of weight $0,$  elements of $\5g_{0},$ also called rotations. The result appears  in Kim-Kol\'a\v{r} \cite{KK}. Here, we propose a different proof.
\bt \label{redr}\cite{KK}
Let $M_0 \subset \mathbb C^{n+1}$ be a polynomial  model hypersurface given by  \eqref{2}, where $Q$ is given by  a sum of squares, that is, $$Q= \sum_{l=1}^k P_l \overline {P_l},$$  where $P_l,  \ l=1, \dots, k, $ are  homogeneous holomorphic polynomials of the same degree. If $M_0$ is holomorphically nondegenerate, then  there is no real diagonal symmetry   and no nilpotent symmetry in $\5g_{0}.$
\et
\bpf Let $X$  be given in its Jordan form after a possible change of coordinates in $z.$ We have  $X= X^{Re}+X^{Im} + X^{Nil}.$   Without loss of generality, we assume that $M_0$ is given by  $\Im w=\sum_{l=1}^k P_l \overline {P_l}.$ Suppose that $ X^{Re} \ne 0.$
 We have $$X^{Re}= \sum \lambda_j z_j \partial_{z_j}, \  \lambda_j \in \Bbb R.$$
We claim that $X^{Re}( P_l )=0,$ for every $l.$ Indeed, if  $P_l$ contains a term $cz^{\alpha},\ c \ne 0,$ then the model contains a term  $Dz^{\alpha} \overline {z^{\alpha}}, \ D > 0. $  
Since $X^{Re}$ is a rotation, it implies    $$ D\sum \lambda_j (2 \alpha_j)=0,$$  and hence $X^{Re}( P_l )=0.$  But it means that $$X^{Re}(\sum_{l=1}^k P_l \overline {P_l})=0,$$ which is impossible since $M_0$ is holomorphically nondegenerate.

Now we deal with the nilpotent part $X^{Nil}.$  Suppose that $ X^{Nil} \ne 0.$ Let $s>1$ be minimal such that ${(X^{Nil})}^s(P_j)=0,$ for all $j.$  We have 
$$\Re {(X^{Nil}(\sum_{j=1}^k P_j \overline {P_j}))}=0.$$ 
Applying ${(X^{Nil})}^{s-1}\overline{(X^{Nil})}^{s-2}$ to the above equation, we obtain
$$\sum_{j=1}^k {(X^{Nil})}^{s-1} ( P_j) \overline{(X^{Nil})^{s-1}  ( {P_j})}=0.   $$
Since this is a sum of square, $(X^{Nil})^{s-1}  ( {P_j})=0,$ for every j. This is  a contradiction with the fact that $s$ was chosen to be minimal.
\epf
We recall the following theorem from \cite{JM} (see also  Proposition 4.1 in  \cite{KK}).
\bt \label{great011}\cite{JM} Let $M_0$ be given by  \eqref{2} be holomorphically nondegenerate with nonvanishing top degree component  $\5g_{1}.$  Suppose  that $\5g_{0}$ contains no  real diagonal rotations. Then $\5g_c=0.$
\et
 If $M_0$ is a sum of squares polynomial model, then the vector field $X$ given by 
$$X:=\sum z_j \partial_{z_j}$$  satisfies   condition (6) on  $\5g_{1}$  in \eqref{grin}. Hence, using Theorem \ref{cris},  Theorem \ref {redr} and  Theorem \ref{great011}, we obtain the following theorem
\bt\label{brno1}\cite{KK}
Let $M_0$  given by  \eqref{2} be a sum of squares polynomial model that is holomorphically nondegenerate. Then    every  element of  $\Aut (M,0),$ where $M$ is any smooth perturbation of $M_0,$ is uniquely determined by its $2$-jet.
\et 

\subsection{The  general case of a pseudoconvex  polynomial  model}\label{bf2}
In this section, we show that pseudoconvex polynomial models do not have either real diagonal rotations. Of course, there are  examples of  pseudoconvex models which are not sums of squares models, as illustrated by:
\be \label{pinson} Let $M_0 \subset  \mathbb C^{2}$ be   given by
 $$\Im w=|z|^8 + |z|^6 {\Re z}^2.$$ 
Notice that $\5g_{1}=0$ in this case.
\ee
We start with the following lemma
\bl\label{infine0+}
Let $F_j$ and $G_j$ be two holomorphic homogeneous polynomials of different degrees, for every $j, \ j=1, \dots, k. $ Then $\sum_{j=1}^k \Re F_j \overline{G_j}$ changes sign or is identically equal to zero.
\el
\bpf
Indeed, for some given  $z_0,$ $$F_j(\lambda z_0)=\lambda^m F_j(z_0), \ \ G_j(\lambda z_0)=\lambda^n G_j(z_0), \ m\ne n.$$ Integrating on the circle $\gamma:=\{ \lambda \in \Bbb C \ | \ |\lambda| =r\},$ we obtain  $\int_\gamma {F_j(\lambda z_0) \overline{G_j(\lambda z_0)} d\lambda}=0.$
\epf
We then have
\bt\label{brno}
Let  $M_0 \subset \Bbb C^{n+1}$ be a pseudoconvex  model hypersurface given by 
 \begin{equation*}
 \Im w =Q(z, \bar z).
 \end{equation*}  If $M_0$ is holomorphically nondegenerate, then there is no  real diagonal  rotation.
\et
\bpf  Suppose that  there is a  symmetry $X$  that is  real diagonal.   Without loss of generality, we assume that $M_0$ is given in   Jordan coordinates for $X,$  so that  $$X=\sum d_j z_j \partial_{z_j},  \ d_j \in \Bbb R.$$ We decompose  $Q$  as 
\begin{equation}\label{bt0}
Q(z, \bar z)= \sum_{j=1}^k c_j z^{\alpha^j} {\overline z}^{\alpha^j} \ +\  \sum_{{\alpha^l} \ne {\alpha^k}} c_{l,k} z^{\alpha^l} {\overline z}^{\alpha^k}. \ \ 
\end{equation}   Since $X$ is a rotation, for any monomial $z^{\alpha^l} \overline{ z^{\alpha^k}}$ in the expansion of $Q,$  we have  \begin{equation}\label{rot}
X(z^{\alpha^l})  +\overline{ X(z^{\alpha^k})}=0.
\end{equation}  It implies that $$X(\sum_{j=1}^k c_j z^{\alpha^j} {\overline z}^{\alpha^j})=0. $$ If there is no second sum, this not possible since  $M_0$ is holomorphically nondegenerate. 
Otherwise, there exists    a nonzero  term  in $z^{\alpha^l} {\overline z}^{\alpha^k}$ that is not a square (in the second sum)  for which  $X(z^{\alpha^l}) \ne 0.$ Hence  ${\overline X} ({\overline z}^{\alpha^k})\ne 0,$ using \eqref{rot}. Computing the Levi form of $M_0$ in the direction of $\overline X, $ we obtain   that the Levi form of the  model $$\Im w=\sum_{{\alpha^l} \ne {\alpha^k}} c_{l,k} z^{\alpha^l} {\overline z}^{\alpha^k}$$  in the direction $\overline X $ is not identically equal to zero. Indeed, its Levi form can  be written as 
$$ \sum_{{\alpha^l} \ne {\alpha^k}} c_{l,k}X( z^{\alpha^l}){\overline X} ({\overline z}^{\alpha^k}).$$ Using Lemma \eqref{infine0+}, this contradicts the assumption that $M_0$ is pseudoconvex.
\epf
 
\begin{Rem}
The existence of nilpotent rotations in the case of general pseudoconvex model hypersurfaces is not known. Nevertheless, they do not occur in $\Bbb C^{3}$ as it is shown by the following theorem.
\end{Rem}
\bt
Let $M_0 \subset \Bbb C^{3}$ be a pseudoconvex model   given by 
 \begin{equation*}
 \Im w =Q(z, \bar z).
 \end{equation*}  If $M$ is holomorphically nondegenerate, then  the nilpotent part of the Jordan form  of a rotation vanishes.
\et
\bpf Suppose by contradiction that $M_0$ admits a nilpotent rotation.  After a possible linear change of coordinates, it is shown in 
 \cite{K21} that  we have, for $z= (z_1, z_2),$
\begin{equation*}
    Q(z, \bar z) = \sum_{j=0}^m
    %_{\mathclap{2j+2l = m}}
    \left(z_1\overline {z_2}+ z_2\overline {z_1}\right)^j Q_{j}(z_2,\overline {z_2}),\ m\ge 1,
\end{equation*} 
 where $ Q_{j}$ is a real valued  polynomial in $(z_2,\overline {z_2})$ of suitable degree.
Computing the Hermitian matrix representing the Levi map of $M_0,$ we obtain 
\begin{equation*}\SMALL  \left( P_{ z_i  \overline{z_k}}\right)=
\begin{pmatrix}
j(j-1)(\Re z_1 \overline {z_2})^{j-2 }z_2 \overline {z_2} Q &  j(j-1)(\Re z_1 \overline {z_2})^{j-2 }z_1 \overline {z_2} Q + j(\Re z_1 \overline {z_2})^{j-1 } Q + j(\Re z_1 \overline {z_2})^{j-1 } \overline {z_2}Q_{\overline {z_2}}\\
%j(j-1)(\Re z_1 \overline {z_2})^{j-2 }z_1 \overline {z_2} Q + j(\Re z_1 \overline {z_2})^{j-1 } Q + j(\Re z_1 \overline {z_2})^{j-1 } \overline {z_2}Q_{\overline {z_2}}
 \dots&j(j-1)(\Re z_1 \overline {z_2})^{j-2 }z_1 \overline {z_1} Q +j(\Re z_1 \overline {z_2})^{j-1 }(\overline {z_1} Q_{\overline {z_2}} +{z_1} Q_{ {z_2}}) +(\Re z_1 \overline {z_2})^{j }Q_{z_2\overline {z_2}}
\end{pmatrix},
\end{equation*}
where $$\dots = \overline{j(j-1)(\Re z_1 \overline {z_2})^{j-2 }z_1 \overline {z_2} Q + j(\Re z_1 \overline {z_2})^{j-1 } Q + j(\Re z_1 \overline {z_2})^{j-1 } \overline {z_2}Q_{\overline {z_2}}}.$$
Computing its determinant, we obtain 
\begin{equation*}
\det {\left( P_{z_i  \overline{z_k}}\right)} = -j^2(j-1) (\Re z_1 \overline { z_2})^{2j-2 }Q^2 
\end{equation*}
$$  
+(\Re z_1 \overline { z_2})^{2j-2 }\left(j(j-1)z_2\overline {z_2}Q Q_{z_2\overline {z_2}} -j^2Q^2
-j^2(z_2 Q Q_{z_2} +\overline {z_2} Q Q_{\overline{z_2}}) -j^2 z_2\overline {z_2} Q_{z_2}Q_{\overline{z_2}}\right)
=$$
$$(\Re z_1 \overline { z_2})^{2j-2 }(-j^2(j-1) Q^2  -j\sum_{\alpha+\beta =d-2j} {2\alpha\beta|c_{\alpha,\beta}|}^2{z_2}^{\alpha+\beta}{\overline{z_2}}^{\alpha+\beta}-
$$ $$ j^2\sum_{\alpha+\beta =d-2j} {(\alpha-\beta)^2|c_{\alpha,\beta}|}^2{z_2}^{\alpha+\beta}{\overline{z_2}}^{\alpha+\beta}-j^2(1+d-2j  )Q^2  -\dots),
$$
where the dots represent an expression  without squares, and hence changing sign  by Lemma \ref{infine0+}. That means that there are points $z_2$ in any neighborhood of $0$ for which this determinant is strictly negative, which means that $M_0$ is not pseudoconvex.

\noindent If now $M_0$ is given by a sum of terms, we notice that the  coefficient of the maximal degree of the determinant in $(\Re z_1 \overline { z_2})^{2j-2 }$  is strictly  negative for a well chosen $z_2$, thanks to the above computation which is valid for every $j.$ Hence  we consider this determinant as a polynomial in $z_1=\overline{z_1}=x_1.$ There is then a solution $x_1$ for which the sign of this determinant  is negative, and therefore since it this is a homogeneous polynomial, there exists a solution $(\lambda x_1, \lambda z_2 )$ as close to zero as needed such that this determinant is strictly negative, which contradicts $M_0$ pseudoconvex at $0.$
\epf
\subsection{Vanishing of $\5g_{c}$: the case of a single chain  in $\mathbb C^{3}$} \label{bf3}
As noticed  for the case of sums of squares models, Theorem \ref{great011} imposes that we have a nonvanishing top degree  component  $\5g_{1}$ to conclude that $\5g_c=0.$  Indeed, using \eqref{grin}, the  key argument of  the proof  is   given by the following nontrivial  equation
$$[Z,Y]=0,$$  valid  for any $Z\in \5g_{1}$ and $Y \in \5g_{c}.$ 

  \noindent Examples constructed from  Example \ref{pinson}  for instance show that  $\5g_{1}=0$  may occur. 
   In \cite{FM16}, the second and third authors gave 
  a characterization  of polynomial models that have nonvanishing $ \5g_c$ in $\Bbb C^3.$

	\bd \cite{FM16}
Let $Y$ be a weighted homogeneous vector field. 
A pair of finite sequences   of holomorphic  weighted homogeneous polynomials $\{U^1, \dots, U^N\}$ and  $\{V^1, \dots, V^N\}$ is called a {\it symmetric pair of $Y$-chains} if 
\begin{equation*}
Y(U^N)=0, \ \ Y(U^j) =c_j U^{j+1}, \ \ j=1, \dots, N-1,
\label{yuv}
\end{equation*}
\begin{equation*}
Y(V^N)=0, \ \ Y(V^j) =d_j V^{j+1}, \ \ j=1, \dots, N-1,
\end{equation*}
where $c_j$, $d_j$ are nonzero complex constants, which satisfy
\begin{equation*} c_j= - \bar d_{N-j}.
\end{equation*}
If the two sequences are identical we say that $\{U^1, \dots, U^N\}$ is a {\it symmetric $Y$-chain}.
\ed

\bt \cite{FM16} Let 
 $M_0$  be a holomorphically nondegenerate hypersurface  given by  \eqref{2},  which admits a nontrivial 
 $Y\in \5g_c$. 
Then $Q$ can be decomposed in the following way

\begin{equation*}
Q=\sum_{j=1}^M T_j,
\end{equation*} where each $T_j$ is given by 
\begin{equation*}
 T_j =\Re \left(\sum_{k=1}^{N_j}
  {U_j^{k}}
  {\overline {{V_j^{N_j -k +1}}}}\right), 
  \end{equation*}
where $\{{ {{U_j^{1}}, \dots, {U_j^{N_j}} }}\}$  and  $\{{ {{V_j^{1}}, \dots, {V_j^{N_j}} }}\}$ are a symmetric pair of $Y$-chains.
 \et

This shows that  the simplest polynomial model  to study  
 for which there exists a nonzero  $Y \in \5g_c$ is given by one single chain, that is, is given by  the following expression
\begin{equation*}
\Im w= \sum_{j=1}^N U^j \overline {U^{N-j+1}},
\end{equation*}
 where
$Y(U^N)=0, Y(U^j)=c_j U^{j+1}, j=1, \dots, N-1, $
 where $c_j$ and  $d_j$  are nonzero complex constants satisfying
  $c_j= - \overline {c_{N-j}}.$
	 	Observe that the  vanishing of  $\5g_1$  may arise  for a single symmetric $Y$-chain, as it is shown in the following example.
\be
Let $M_0 \subset \Bbb C^{3}$ be  given by 
\begin{equation*}
\Im w= P\bar Q + Q\bar P,
\end{equation*}
where $P$ and $Q$ are given by 
\begin{equation*}
\begin{aligned}
P(z_1, z_2) &= i z_1^2 z_2^3 (z_1 - z_2)
\\
Q(z_1, z_2) &=  3 z_1^3 z_2^5 (z_1 - z_2)
\end{aligned}
\end{equation*}
and  let $Y$ be  given by 
\begin{equation*}
 Y =z_1 z_2^2 (5z_1 - 6z_2) \d_{z_1}  - z_2^3 (4z_1 - 3 z_2)
\d_{z_2}.
\end{equation*} One can check that $Y \in \5g_c$ and 
 $\5g_1=\{0\}.$
\ee
We have the following theorem that shows that pseudoconvex models have vanishing $\5g_c$:
\bt \label{capi}
 Let  $M_0 \subset \Bbb C^{3}$ be  holomorphically nondegenerate with one single  symmetric Y -chain. Then  $M_0$ is not pseudoconvex.
\et
\bpf
Assume first that the single chain is of length two, that is $M_0$ is given by 
\begin{equation*}
\Im w= P\bar Q + Q\bar P.
\end{equation*}
Computing the Hessian matrix, we get
$$
\begin{pmatrix}
P_{z_1}\overline {Q}_{\bar {z_1}} + Q_{z_1}\bar P_{\bar {z_1}} & P_{z_1}\overline {Q}_{\bar {z_2}} + Q_{z_1}\bar P_{\bar {z_2}}\\
%j(j-1)(\Re z_1 \overline {z_2})^{j-2 }z_1 \overline {z_2} Q + j(\Re z_1 \overline {z_2})^{j-1 } Q + j(\Re z_1 \overline {z_2})^{j-1 } \overline {z_2}Q_{\overline {z_2}}
 P_{z_2}\overline {Q}_{\bar {z_1}} + Q_{z_2}\bar P_{\bar {z_1}}&P_{z_2}\overline {Q}_{\bar {z_2}} + Q_{z_2}\bar P_{\bar {z_2}}
\end{pmatrix}.
$$ 
Computing the determinant of this matrix, we find that it is equal to
$$
(P_{z_1}\overline {Q}_{\bar {z_1}}+ Q_{z_1}\bar P_{\bar {z_1}})(P_{z_2}\overline {Q}_{\bar {z_2}} + Q_{z_2}\bar P_{\bar {z_2}})-(P_{z_1}\overline {Q}_{\bar {z_2}} + Q_{z_1}\bar P_{\bar {z_2}})(P_{z_2}\overline {Q}_{\bar {z_1}} + Q_{z_2}\bar P_{\bar {z_1}})=
$$
$$
P_{z_1}\overline {Q}_{\bar {z_1}}P_{z_2}\overline {Q}_{\bar {z_2}}+P_{z_1}\overline {Q}_{\bar {z_1}}Q_{z_2}\bar P_{\bar {z_2}}+
 Q_{z_1}\bar P_{\bar {z_1}}P_{z_2}\overline {Q}_{\bar {z_2}} +  Q_{z_1}\bar P_{\bar {z_1}}Q_{z_2}\bar P_{\bar {z_2}}-
$$
$$-
P_{z_1}\overline {Q}_{\bar {z_2}}P_{z_2}\overline {Q}_{\bar {z_1}}-P_{z_1}\overline {Q}_{\bar {z_2}}Q_{z_2}\bar P_{\bar {z_1}}-
Q_{z_1}\bar P_{\bar {z_2}}P_{z_2}\overline {Q}_{\bar {z_1}}-Q_{z_1}\bar P_{\bar {z_2}}Q_{z_2}\bar P_{\bar {z_1}}=
$$
$$
P_{z_1}\overline {Q}_{\bar {z_1}}Q_{z_2}\bar P_{\bar {z_2}}+
 Q_{z_1}\bar P_{\bar {z_1}}P_{z_2}\overline {Q}_{\bar {z_2}} -
$$
$$-P_{z_1}\overline {Q}_{\bar {z_2}}Q_{z_2}\bar P_{\bar {z_1}}-
Q_{z_1}\bar P_{\bar {z_2}}P_{z_2}\overline {Q}_{\bar {z_1}}= 
$$
$$
-(P_{z_1} Q_{z_2} -P_{z_2} Q_{z_1})\overline{(P_{z_1} Q_{z_2} -P_{z_2} Q_{z_1})}\le 0.$$
Since $M_0$ is holomorphically nondegenerate, this determinant is not identically $0.$
Now, if the  single symmetric Y-chain is of any length, the determinant of the Hessian matrix will be composed, by expanding the determinant, of determinants of matrices of the form
$$
\begin{pmatrix}
R_{z_1}\overline {T}_{\bar {z_1}} + R_{z_1}\bar T_{\bar {z_1}} & S_{z_1}\overline {U}_{\bar {z_2}} + S_{z_1}\bar U_{\bar {z_2}}\\
%j(j-1)(\Re z_1 \overline {z_2})^{j-2 }z_1 \overline {z_2} Q + j(\Re z_1 \overline {z_2})^{j-1 } Q + j(\Re z_1 \overline {z_2})^{j-1 } \overline {z_2}Q_{\overline {z_2}}
 R_{z_2}\overline {T}_{\bar {z_1}} + R_{z_2}\bar T_{\bar {z_1}}&S_{z_2}\overline {U}_{\bar {z_2}} + S_{z_2}\bar U_{\bar {z_2}}
\end{pmatrix},
$$
where $R\bar T+ T\bar R$  and $S\bar U+ U\bar S$ are  parts of the chain,
and one (if possible) of the form 
$$
\begin{pmatrix}
L_{z_1}\overline {L}_{\bar {z_1}}  & L_{z_1}\overline {L}_{\bar {z_2}} \\
%j(j-1)(\Re z_1 \overline {z_2})^{j-2 }z_1 \overline {z_2} Q + j(\Re z_1 \overline {z_2})^{j-1 } Q + j(\Re z_1 \overline {z_2})^{j-1 } \overline {z_2}Q_{\overline {z_2}}
 L_{z_2}\overline {L}_{\bar {z_1}} &L_{z_2}\overline {L}_{\bar {z_2}} 
\end{pmatrix},
$$
if $L\bar L$ is part of the chain.
We then see that if $(R,T)$ is different from $(S,U)$ in the chain,  then there are no squares in those determinants, and hence it may change signs by Lemma \ref{infine0+}. The last determinant is identically equal to $0.$ Therefore the only determinants that matter are when $(R,T)=(S,U)$, and all
those determinants are negative, as we saw in the first part of the proof. We note that for the element containing the last element of the chain, the corresponding determinant is not identically to zero, by construction, and hence the conclusion.
\epf
\begin{Rem}
In a forthcoming paper, the second and third authors with Liczman will show that Theorem \ref{capi} holds in full generality, namely,  models with nonvanishing $\5g_c$ are not pseudoconvex in $\Bbb C^{3}.$
 We would like to emphasize that the following lemma is crucial, when working  with models in $\Bbb C^{3}:$
\bl {Let $Y=f_1 (z_1, z_2)\frac{\partial }{\partial z_1} +f_2  (z_1, z_2)\frac{\partial }{\partial z_2} $ be a   nonzero 
homogeneous holomorphic  vector field of weighted degree $> 0$.  Then the space of weighted homogeneous polynomials  
in $z=(z_1,z_2)$ of a given weighted degree $\nu$
annihilated by $X$ has complex dimension at most one.}
\el
The question whether possible nonvanishing   $\5g_c$ for pseudoconvex models (with top degree component $\5g_1$ possibly vanishing)  holds,  is still open for $n>2.$
\end{Rem}

 \section{Part II: The  $1$-jet determination of  stationary discs attached to generic CR submanifolds: new examples and their consequences}

Let $M \subset \C^{n+d}$ be a  $\mathcal{C}^{4}$ generic real submanifold of real codimension $d\ge 1$ through $p=0$ given locally by the following system of equations 
\begin{equation}\label{eqred0}
\begin{cases}
 \Re   w_1= \transp\bar z A_1 z+ O(3)\\
\ \ \ \ \vdots \\
\Re   w_d = \transp\bar z A_d z+ O(3)
\end{cases}
\end{equation}
where $A_1,\hdots,A_d$ are Hermitian matrices of size $n$. In the remainder O(3), the variables $z$ and $\Im w$  are respectively of weight one and two. We denote by $r_1,\ldots,r_d$ the above defining functions, we set $r=(r_1,\ldots,r_d)$ and we write $M=\{r=0\}$.     
We associate to $M$ its corresponding  model quadric $M_H$ given by
\begin{equation}\label{eqred1}
\begin{cases}
\Re  w_1 = \transp\bar z A_1 z\\
\ \ \ \ \vdots \\
\Re  w_d = \transp\bar z A_d z.
\end{cases}
\end{equation}

In a recent series of papers \cite{be-bl-me,tu,be-me,be-me2,be-me3} the authors have investigated geometric conditions on $M$, and in particular on its model  $M_H$, that ensure that (germs at $0$ of) CR automorphisms of $M$ fixing the origin are uniquely 
determined by their $2$-jet at the origin. It is important to point out that in each of these papers, the authors have focused on the method of stationary discs which has proven to be well adapted for such problems.

\subsection{Nondegenerate submanifolds}
The submanifold $M$ given by \eqref{eqred0} is {\it strongly Levi nondegenerate at $0$} (resp. {\it strongly pseudoconvex at $0$}) if there exists 
$b \in \Bbb R^d$ such that the matrix $\sum_{j=1}^d b_jA_j$ 
is invertible (resp. positive definite).  We say   $M$ is {\it $\mathfrak{D}$-nondegenerate} at 0 if, in addition to being strongly Levi nondegenerate at $0$, 
there  exists $V \in \C^{n}$ such that the $d \times d$ matrix  
$$\Re\left(\transp \overline D \left(\sum_{j=1}^d b_jA_j\right)^{-1}D\right)$$ 
is invertible, where $D$ is the $n \times d$ matrix whose $j^{th}$ column is
 $A_jV$. Note that in such case we have $d\leq 2n$.  

We now discuss a matrix equation associated with $M$ that plays an important role in the understanding of stationary discs attached to $M$. 
Assume that  $M$ is strongly Levi nondegenerate with $b\in \R^d$ such that $\sum_{j=1}^d b_jA_j$  is invertible. 
   For  $a \in \C^d$ sufficiently small, we define  
   $$P:=\sum_{j=1}^d {{a}_j} A_j  \ \ \ \ \ \ \     A:=\sum_{j=1}^d (b_j -a_j-\overline{a_j})A_j,$$ and we 
 consider the  following  matrix equation 
\begin{equation}\label{eqX}
PX^2 + AX +   \transp{\overline {P}}=0.
\end{equation} 
Let  $X$ be the unique $n\times n$ matrix solution 
  of \eqref{eqX} with $\|X\|<1.$  
We point out that $X$ depends on both $a$ and $b$. Moreover, $X$ is real analytic with respect to $(a, \bar a).$
The following lemma is new and important for our approach.

\begin{Lem}\label{lemdiffX}
We have, for $s=1,\ldots,d$,
 \begin{equation}\label{eqbre}
 X_{\Re a_s}= \sum_{r=0}^\infty (-1)^{r+1}\left((I+A^{-1}P X)^{-1}A^{-1}P\right)^r(I+A^{-1}P X)^{-1}A^{-1} A_s(I-X)^2X^r
  \end{equation}
  where $X_{\Re  a_s}$ denotes the partial derivative of $X$ with respect to $\Re  a_s$. 
\end{Lem}

\begin{proof}

 Differentiating  \eqref{eqX}   in the variable $\Re{ a_s}$ and evaluating at $(a,b- a-\overline{a})$ leads to 
$$A_s X^2 + P (XX_{\Re a_s}+X_{\Re a_s}X)-2A_sX+AX_{\Re a_s}+A_s=0.$$
Thus
$$AX_{\Re a_s}+P (XX_{\Re a_s}+X_{\Re a_s}X) = -A_s(I-X)^2$$
$$X_{\Re a_s}+A^{-1}P (XX_{\Re a_s}+X_{\Re a_s}X) = -A^{-1} A_s(I-X)^2$$
$$(I+A^{-1}P X)X_{\Re a_s}+A^{-1}PX_{\Re a_s}X = -A^{-1} A_s(I-X)^2$$
$$(I+A^{-1}P X)X_{\Re a_s}= -A^{-1} A_s(I-X)^2-A^{-1}PX_{\Re a_s}X $$
which finally implies
$$X_{\Re a_s}= -(I+A^{-1}P (X)^{-1}A^{-1} A_s(I-X)^2-(I+A^{-1}P X)^{-1}A^{-1}PX_{\Re a_s}X $$
The result follows by iterations.
\end{proof}

Following \cite{be-me2}, we say that $M$ is {\it stationary minimal at $0$ for $(a,b -a-\overline{a},V)$}, where $V\in \C^n$, if the $d$ matrices $A_1,\ldots,A_d$ restricted to the orbit space
 $$\mathcal{O}_{X, V}:={\rm span}_\R \{V, XV, {X^2}V, \ldots, {X^k}V, \ldots\}$$
 are $\R$-linearly independent. 
We recall that if $M$ is $\mathfrak{D}$-nondegenerate then it is stationary minimal for $(0,b,V)$. 

\subsection{Stationary discs and their $1$-jet determination}
Let $M=\{r=0\}\subset \C^{n+d}$ be a $\mathcal{C}^{4}$ generic  real submanifold of codimension $d$ given by  (\ref{eqred0}). Following Lempert \cite{le} and Tumanov \cite{tu},  a holomorphic disc $f: \Delta \to \C^{n+d}$ continuous up to  $\partial \Delta$ and such that $f(\partial \Delta) \subset M$ is {\it stationary for $M$} if there 
exist  $d$ real valued functions $c_1, \ldots, c_d: \partial \Delta \to \R$ with $\sum_{j=1}^dc_j(\zeta)\partial r_j(0)\neq 0$ for all $\zeta \in \partial \Delta$   and such that the map denoted by $\tilde{f}$
\begin{equation*}
\zeta \mapsto \zeta \sum_{j=1}^dc_j(\zeta)\partial r_j\left(f(\zeta), \overline{f(\zeta)}\right)
\end{equation*}
defined on $\partial \Delta$ extends holomorphically on $\Delta$. In that case, the disc $\bm{f}=(f,\tilde{f})$ is a holomorphic lift of $f$ to the cotangent bundle $T^*\C^{N}$ and the set of all such lifts $\bm{f}=(f,\tilde{f})$, with $f$ nonconstant, is denoted by $\mathcal{S}(M)$. We also consider lifts attached to a fixed point in the cotangent bundle; we fix $b_0 \in  \Bbb R^d$  such that $\sum_{j=1}^d {b_0}_jA_j$ is invertible and
 we define $\mathcal{S}_0(M_H) \subset \mathcal{S}(M_H)$ to be the subset of lifts whose value at $\zeta=1$ is $(0,0,0,b_0/2)$.

It is remarkable that the strong Levi nondegeneracy of $M$ allows to construct lifts of stationary discs by small perturbations of an initial one (see e.g. Theorem 2.1 in \cite{be-bl-me}). We note that the choice of the initial stationary disc is quite important in this approach. However strong Levi nondegeneracy alone does not guarantee that 
the constructed family of discs shares properties that can be used for the jet determination of CR automorphisms. More precisely, we care about the $1$-jet determination of discs (see e.g. Proposition 3.6 in \cite{be-bl-me}) and a certain filling property by those discs 
(see e.g. Proposition 3.3 in  \cite{be-bl-me}). We noticed in \cite{be-me} that $1$-jet determination of discs is a key property that actually implies a relevant filling property (Proposition 4.4 in \cite{be-me}, see also corollary 5.6 in \cite{tu}). 
Starting \cite{be-me2}, we then focused on the $1$-jet determination of discs and turned to a different approach in which we describe explicit lifts of stationary discs for the model quadric $M_H$ near a given one (Theorem 3.1 in \cite{be-me2}; see also Proposition 4.2 in \cite{tu} in case $M_H$ is strongly pseudoconvex). We briefly discuss the $1$-jet determination of stationary discs.  
 
Let $M_H \subset \Bbb C^{n+d}$ be a model submanifold of codimension $d$ given by \eqref{eqred1}. We consider the  $1$-jet map at $\zeta=1$ defined on $\mathcal{S}_0(M_H)$ 
by  
$$\mathfrak j_{1}:\bm{f} \mapsto (\bm{f}(1),\bm{f}'(1)).$$
Since $\bm{f}(1)=(0,0,0,b_0/2)$ with $b_0 \in \R^d$ fixed, we identify $\mathfrak j_{1}$ with the derivative map  
$\bm{f} \mapsto \bm{f}'(1)$. In the paper \cite{be-me2}, using the explicit expression of discs, we proved  that                                                                       
the $1$-jet map $\mathfrak j_{1}: \C^d \times \C^n \to  \C^n \times \R^d \times \R^d$  can essentially be written as 
\begin{equation*} 
\mathfrak j_{1}
: (a,V) \mapsto 
(V, \transp\overline{V}(I-\transp\overline {X})K_j(I-X)V, \Im  {a})
\end{equation*}
where $X$ is the small solution of \eqref{eqX} and, for $j=1, \dots,d$, 
\begin{equation}\label{eqK_j}
K_j=\sum_{r=0}^\infty {\transp \overline {X}}^r A_j X^r. 
\end{equation}
We then recall the following key lemma
\begin{Lem}\label{lemjet}\cite{be-me2}  Let $M_H$ be a strongly Levi nondegenerate quadric given by \eqref{eqred1} and let $b \in \R^d$ be such that $\sum_{j=1}^d b_jA_j$ is invertible.
 Let $a \in \C^d$ be small enough and let  X be the unique $n\times n$ matrix solution of \eqref{eqX} with  $\|X\|<1$.  
 \begin{enumerate}[i.]
\item  The $1$-jet map $\mathfrak j_{1}$ is a local diffeomorphism at $(a,V) \in \C^d\times \C^n$ if and only if 
the  $d\times d$ matrix 
$$\left(\dfrac{\partial}{\partial {\Re  a_s}}\transp\overline{V}(I-\transp\overline {X})K_j(I-X)V\right)_{j,s}$$ is invertible.
\item
For any $s=1,\ldots,d$, we have 
\begin{eqnarray*}
\dfrac{\partial}{\partial {\Re a_s}}\transp\overline{V}(I-\transp\overline {X})K_j(I-X)V& = & -2\Re  \left(\sum_{r=0}^{\infty} \transp\overline{V}\left(I-\transp \overline {X}\right)^2 {\transp \overline {X}}^rK_jX_{\Re  a_s}X^rV\right).\\
\end{eqnarray*}
\end{enumerate}
\end{Lem}

Using this  lemma, we proved that if a model $M_H$ is strongly pseudoconvex (with  $b \in \R^d$ such that $\sum_{j=1}^d b_jA_j$ is positive definite)
and stationary minimal at $0$ for $(a,b-a-\overline{a},V)$, with $a \in \C^d $ sufficiently small and $V \in \C^n$, then the  $1$-jet map $\mathfrak j_{1}$ is a local diffeomorphism at $(a,V)$  (see Theorem 4.3 in \cite{be-me3}). The proof of this result strongly suggests a possible extension to strongly Levi nondegenerate model quadrics. Accordingly we defined in \cite{be-me3}
\bd\label{definondeg100}
  Let $M\subset \C^{n+d}$ be a real submanifold given by \eqref{eqred0}. Assume $M$ is strongly Levi nondegenerate at $0$ and let $b\in \R^d$ be such that $\sum_{j=1}^d b_jA_j$ is invertible. Let $a \in \C^d$ be sufficiently small and set  $A:=\sum_{j=1}^d (b_j -a_j-\overline{a_j})A_j.$  
 We say that  $M$ is  {\it  $\mathfrak{D}(a)$-nondegenerate} at $0$  if there exists $V\in \C^n$ such that  the  matrix 
\begin {equation*}
\Re\left(\sum_{r=0}^{\infty} \transp \overline V{\transp \overline {X}}^rA_j A^{-1}A_sX^r V\right)_{j,s}
\end {equation*}
is nondegenerate.
\ed
We point our that when $a = 0$ we recover the definition of  $\mathfrak{D}$-nondegeneracy and 
if  $M$ is  $\mathfrak{D}(a)$-nondegenerate at $0$  then it is stationary minimal at $0$ for $(a,b -a-\overline a,V)$ for some $V\in \C^n$.  We  also emphasize that, unlike the  notion of $\mathfrak{D}$-nondegeneracy  which imposes $d\le 2n,$ the
 $\mathfrak{D}(a)$-nondegeneracy  condition  does not impose any restriction on the codimension.

We conjectured in \cite{be-me3} that if $M$ is   $\mathfrak{D}(a)$-nondegenerate at $0$ then the $1$-jet map $\mathfrak j_{1}$ a local diffeomorphism at $(a,V)$ for some 
$V \in \C^n$, and that, accordingly, (germs at $0$ of)  CR automorphisms  of $M$ are  uniquely determined by their $2$-jet at $0$.    
While we are not yet a stage where we can settle this conjecture, we provide in the next section an example of a $\mathfrak{D}(a)$-nondegenerate quadric which is 
neither strongly pseudoconvex nor $\mathfrak{D}$-nondegenerate and for which the $1$-jet map $\mathfrak j_{1}$ is a local diffeomorphism at $(a,V)$. This example (and its perturbations) is 
then the first known example where $2$-jet determination of  CR automorphisms holds outside of the realm of previously known submanifolds, namely  $\mathfrak{D}$-nondegenerate ones \cite{be-me} and strongly pseudoconvex ones \cite{tu3}.

\subsection{ New  examples}
We have 
\begin{Thm}\label{theoex}
Let $M \subset \C^{10}$ be the  real submanifold of  codimension $7$ through $0$ given   by 
\begin{equation}\label{eqFe}
\begin{cases}
\Re w_1= z_1 \overline{ z_1} +  z_2 \overline{ z_2} - z_3 \overline{ z_3} \\
\Re w_2=  z_1 \overline{ z_1} - z_2 \overline{ z_2} \\
\Re w_3= z_2 \overline{ z_3} + z_3 \overline{ z_2}  \\
\Re w_4= z_1 \overline{ z_2}  + z_2 \overline{ z_1} \\
\Re w_5=iz_1 \overline{ z_3 } - iz_3 \overline{ z_1 } \\
\Re w_6=iz_2 \overline{ z_3 } - iz_3 \overline{ z_2 } \\
\Re w_7=iz_1 \overline{ z_2 } - iz_2 \overline{ z_1 } \\
\end{cases}
\end{equation}
 Then   $M$ is not strongly pseudoconvex nor $\mathfrak{D}$-nondegenerate at 0. Moreover there exists $a\ne 0$  small enough and $V\in \C^3$  such that  
 \begin{enumerate}[i.]
\item  $M$   is $\mathfrak{D}(a)$-nondegenerate at $0$, and,
\item  the $1$-jet map at $\zeta=1$ defined on $\mathcal{S}_0(M)$ by $\mathfrak j_{1}:\bm{f} \mapsto \bm{f}'(1)$
is a diffeomorphism at (the lift parametrized by) $(a,V)$.
 \end{enumerate}
\end{Thm}
As a direct corollary (see e.g. Theorem 3.7 in \cite{be-me}), we have
\begin{Cor}
Let  $M$ be the submanifold given in Theorem  \ref{theoex}. Then the germs  at $0$ of  CR automorphisms  of $M$ of class $\mathcal{C}^3$  are uniquely determined by their $2$-jet at $0$. Furthermore, the result holds for  
$\mathcal{C}^4$ perturbations of $M$ given by \eqref{eqred0}.  
\end{Cor}

\begin{proof}[Proof of Theorem \ref{theoex}]
We first note that due to its codimension, the submanifold $M$ is not $\mathfrak{D}$-nondegenerate at 0. Now, the matrices corresponding to \eqref{eqFe} are
$$A_1=\left(\begin{array}{ccc}1&0&0\\0&1&0\\0&0&-1\end{array}\right)\qquad 
A_2=\left(\begin{array}{ccc}1&0&0\\0&-1&0\\0&0&0\end{array}\right)\qquad 
A_3=\left(\begin{array}{ccc}0&0&0\\0&0&1\\0&1&0\end{array}\right)\qquad 
A_4=\left(\begin{array}{ccc}0&1&0\\1&0&0\\0&0&0\end{array}\right)\qquad$$ 
$$
A_5=\left(\begin{array}{ccc}0&0&i\\0&0&0\\-i&0&0\end{array}\right)\qquad 
A_6=\left(\begin{array}{ccc}0&0&0\\0&0&i\\0&-i&0\end{array}\right)\qquad 
A_7=\left(\begin{array}{ccc}0&i&0\\-i&0&0\\0&0&0\end{array}\right).
$$
Note that 
$${\rm span}_\R \{A_1,\ldots, A_7 \} =\left\{\left(\begin{array}{cccc}\lambda_1 + \lambda_2&\lambda_4 + i\lambda_7& i\lambda_5 \\\lambda_4 - i\lambda_7&\lambda_1 -\lambda_2&\lambda_3 + i\lambda_6\\-i\lambda_5&\lambda_3 -i\lambda_6&-\lambda_1\end{array}\right) \ | \ \lambda_j \in \R, j=1,\ldots,7 \right\},$$
from which it follows that the matrices $A_j$'s are linearly independent. Moreover $M$ is strongly Levi nondegenerate at $0$ but is not strongly pseudoconvex since the positivity of the diagonal terms $\lambda_1 + \lambda_2>0,$  $\lambda_1 -\lambda_2>0$ and $-\lambda_1>0$ implies $\lambda_1 >0$ and $\lambda_1 <0.$

\vspace{0.5cm }
We now prove i. We will show that $M$ is $\mathfrak{D}(a)$-nondegenerate at $0$ for $a=(\varepsilon, \varepsilon,0, \ldots, 0)$ for some small 
enough $\varepsilon>0$. Define for $\varepsilon>0$
 $$A:=A_1 \ \mbox{ and } \ P:= \varepsilon (A_1 + A_2)= \varepsilon \left(\begin{array}{cccc}2&0&0\\0&0&0\\0&0&-1\end{array}\right).$$
%Note that
%$$ A^{-1}P=\varepsilon \left(\begin{array}{ccc}2&0&0\\0&0&0\\0&0&1\end{array}\right).$$
The $3\times 3$ matrix  $X$ solution of \eqref{eqX} with $\|X\|<1$ is given by
$$ X= \left(\begin{array}{cccc}\alpha&0&0\\0&0&0\\0&0&\gamma\end{array}\right),$$
where $\alpha$ and $\gamma$ are solutions of quadratic equations
$$2\varepsilon \alpha^2 + \alpha + 2\varepsilon=0 \ \  \mbox{ and } \ \ \varepsilon \gamma^2 +\gamma  +\varepsilon=0.$$
We will now show that for  $$V=\left(\begin{array}{c}1\\1\\1\end{array}\right)$$
 the  matrix 
$$ \Re\left(\sum_{r=0}^{\infty} \transp \overline V{\transp \overline {X}}^rA_j A^{-1}A_sX^r V\right)_{j,s}$$
is nondegenerate. Using the Vandermonde determinant, since $\alpha \ne  \gamma,$ we have for  $\varepsilon$ small enough
$${\rm span}_\R \{V, XV, {X^2}V, \ldots, {X^r}V, \ldots\}=\C^3.$$ 
For $r>0,$ we computing  $A_j(X^rV)$ and obtain 
$$A_1 X^rV= \left(\begin{array}{c}\alpha^r\\0\\-\gamma^r\end{array}\right) \ \ A_2 X^rV= \left(\begin{array}{c}\alpha^r\\0\\0\end{array}\right) \ \ A_3 X^rV= \left(\begin{array}{c}0\\\gamma^r\\0\end{array}\right) \ \ A_4 X^rV= \left(\begin{array}{c}0\\\alpha^r\\0\end{array}\right)$$
$$A_5 X^rV= \left(\begin{array}{c}i\gamma^r\\0\\-i\alpha^r\end{array}\right) \ \ A_6 X^rV= \left(\begin{array}{c}0\\i\gamma^r\\0\end{array}\right) \ 
\ A_7 X^rV= \left(\begin{array}{c}0\\-i\alpha^r\\0\end{array}\right).$$
Therefore, after a direct computation, we obtain
\begin {equation*}
\Re\left(\sum_{r=0}^{\infty} \transp \overline V{\transp \overline {X}}^rA_j A^{-1}A_sX^r V\right)_{j,s}= \left(\begin{array}{cc}B_1&(0)\\(0)&B_2\end{array}\right)
\end {equation*}
with 
\begin{equation}\label{eqB_1}
B_1= \left(\begin{array}{ccccccc}1+\dfrac{\alpha^2}{1- \alpha^2}-\dfrac{\gamma^2}{1- \gamma^2} &\dfrac{\alpha^2}{1- \alpha^2}&2&2\\
\dfrac{\alpha^2}{1- \alpha^2}&2+\dfrac{\alpha^2}{1- \alpha^2}&-1&0\\
2&-1&\dfrac{\gamma^2}{1- \gamma^2}&1+\dfrac{\alpha \gamma}{1- \alpha \gamma}\\
2&0&1+\dfrac{\alpha \gamma}{1- \alpha \gamma}&2+ \dfrac{\alpha^2}{1- \alpha^2}\\
\end{array}\right)
\end {equation}
 and
%$$B_2=\left(\begin{array}{ccc}
%\gamma^2-\alpha^2 + \gamma^4 -\alpha^4 +\dots&-1&1\\-1&\gamma^2+\gamma^4 +\dots&-1 -\alpha\gamma-\alpha^2\gamma^2 +\dots\\1&-1-\alpha\gamma-\alpha^2\gamma^2 +\dots&2+\alpha^2 +\alpha^4 +\dots\end{array}\right).$$
\begin{equation}\label{eqB_2}
B_2=\left(\begin{array}{ccc}
\frac{1}{1-\gamma^2} -\frac{1}{1-\alpha^2}&-1&1\\
-1&-1+\frac{1}{1-\gamma^2} &-\frac{1}{1-\alpha\gamma}\\
1&-\frac{1}{1-\alpha\gamma}&1+\frac{1}{1-\alpha^2}\end{array}\right).
\end {equation}
The determinant of the matrix $B_1$ is a rational function with respect to $\alpha$ and $\gamma$. After putting the terms on each column on a common denominator, this determinant equals the one of the matrix  
%after   of the form $$\dfrac{R (\alpha, \gamma)}{(1- \alpha^2)^3(1- \gamma^2)^2(1- \alpha^2) (1- \alpha\gamma)^2} .$$
%We claim that $R (\alpha, \gamma)$
%contains    a  nonzero monomial term  of maximal degree  $ 14.$ 
%Indeed, it can be checked that $R (\alpha, \gamma)$ is the determinant of the matrix 
$$ \left(\begin{array}{cccc}1-2 \gamma^2+ \alpha^2\gamma^2 &\alpha^2&2 -2 \alpha\gamma -2\gamma^2 +2 \alpha\gamma^3&2 -2 \alpha\gamma -2\alpha^2 +2 \alpha^3\gamma\\\alpha^2-\alpha^2\gamma^2&2-\alpha^2&-1+ \alpha\gamma +\gamma^2 -\alpha\gamma^3 &0\\2 -2\gamma^2 -2\alpha^2 +2\alpha^2\gamma^2&-1+ \alpha^2 &\gamma^2-\alpha\gamma^3&1-\alpha^2\\ 2 -2\gamma^2 -2\alpha^2 +2\alpha^2\gamma^2&0&1-\gamma^2 &2-\alpha^2 -2 \alpha\gamma + \alpha^3\gamma \end{array}\right),$$  with maximal degree term  equal to  $7\alpha^8\gamma^6.$
Picking the nonzero minimal order homogeneous polynomial in the development   of this polynomial, we obtain a minimal order term with respect to $\varepsilon.$ Since $\varepsilon\neq0$ may be chosen as small as desired,   we obtain that this determinant is nonzero, due to the form of the solutions $\alpha$ and $\gamma.$ The same argument applied to the matrix $B_2$ shows that its determinant is nonzero for (the same value of) $\varepsilon\neq0$ small enough.   
%\textcolor{red}{THIS SHOULD BE FIXED. DO YOU MIND DOING SO? worst cae scenario we can talk about the explicit values we used. Computing the determinant of $B_2$ 
%$$
%\left(\begin{array}{ccc}
%\gamma^2-\alpha^2+ \gamma^4 -\alpha^4&-1&1\\-1&\gamma^2+\gamma^4&-1 -\alpha\gamma-\alpha^2\gamma^2\\1&-1-\alpha\gamma-\alpha^2\gamma^2&2+\alpha^2 +\alpha^4\end{array}\right),$$ we notice that there is a nonzero  homogeneous polynomial of (minimal) order $2$ in the development of this determinant, a   polynomial with respect to $\alpha$ and $\gamma,$  of the form
%$$2\alpha\gamma -2\gamma^2.$$
%Therefore the desired determinant is nonzero as claimed, since $\alpha\ne\gamma$. Does this really work? This is not $B_2$ but a rough approximation.}  
This proves that $M$ is   $\mathfrak{D}(a)$-nondegenerate at $0,$ for $a=(\varepsilon, \varepsilon,0, \dots, 0)$ with $\varepsilon  \ne 0$ small enough.
%Finally, note that, since $d > 2n$ (here $d=7$ and $n=3$), this determinant is zero for $\varepsilon=0.$ 

\vspace{0.5cm} 

We now move to the second statement $ii.$ and show that the  $1$-jet map $\mathfrak j_{1}: \C^3 \times \C^7 \to  \C^7 \times \R^3 \times \R^3$ written as 
\begin{equation*} 
\mathfrak j_{1}
: (a,V) \mapsto 
(V, \transp\overline{V}(I-\transp\overline {X})K_j(I-X)V, \Im {a}).
\end{equation*}
is a diffeomorphism at $(a,V)=((\varepsilon, \varepsilon,0, \ldots, 0),(1,1,1))$. According to Lemma \ref{lemjet}, we then need to show that the $7 \times 7$ matrix 
$$\Re\left(\sum_{r=0}^{\infty} \transp\overline{V}\left(I-\transp \overline {X}\right)^2 {\transp \overline {X}}^rK_jX_{\Re a_s}X^rV\right)$$
is invertible at $(a,V)=((\varepsilon, \varepsilon,0, \ldots, 0),(1,1,1))$. First, note that for $r>0$ 
$$X^rV=\left(\begin{array}{c}\alpha^r\\\delta_{0r}\\\gamma^r\end{array}\right) \ \mbox{ and } \ \transp\overline{V}\left(I-\transp \overline {X}\right)^2 {\transp \overline {X}}^r=(\alpha^r(1-\alpha)^2,\delta_{0r},\gamma^r(1-\gamma)^2),$$
%while 
%$$\transp\overline{V}\left(I-\transp \overline {X}\right)^2=((1-\alpha)^2,1,(1-\gamma)^2).$$
Here we have used the Kronecker delta. 
Moreover 
$$K_1 =\left(\begin{array}{ccc}\frac{1}{1-\alpha^2} &0&0\\0&1&0\\0&0&\frac{-1}{1-\gamma^2} \end{array}\right) \ K_2 =\left(\begin{array}{ccc}\frac{1}{1-\alpha^2} &0&0\\0&-1&0\\0&0&0 \end{array}\right) \ K_3=\left(\begin{array}{ccc}0&0&0\\0&0&1\\0&1&0\end{array}\right) \ K_4 = \left(\begin{array}{ccc}0&1&0\\1&0&0\\0&0&0\end{array}\right) $$
$$K_5 =\left(\begin{array}{ccc}0 &0&\frac{i}{1-\alpha\gamma}\\0&0&0\\\frac{-i}{1-\alpha\gamma}&0&0 \end{array}\right) \ K_6 =\left(\begin{array}{ccc}0&0&0\\0&0&i\\0&-i&0\end{array}\right) \ K_7 =\left(\begin{array}{ccc}0&i&0\\-i&0&0\\0&0&0\end{array}\right).$$
Then for $r\geq 0$ we have 
$$ (\transp\overline{V}\left(I-\transp \overline {X}\right)^2 {\transp \overline {X}}^rK_j)_j=\left(\begin{array}{ccc}
\alpha^r\frac{1-\alpha}{1+\alpha} & \delta_{0r} &-\gamma^r\frac{1-\gamma}{1+\gamma}\\
\alpha^r\frac{1-\alpha}{1+\alpha} & -\delta_{0r}& 0\\
0&\gamma^r(1-\gamma)^2&\delta_{0r}\\
\delta_{0r}&\alpha^r(1-\alpha)^2&0\\
-i\gamma^r\frac{(1-\gamma)^2}{1-\alpha\gamma}&0&i\alpha^r\frac{(1-\alpha)^2}{1-\alpha\gamma}\\
0&-i\gamma^r(1-\gamma)^2&i\delta_{0r}\\
-i\delta_{0r}&i\alpha^r(1-\alpha)^2&0\\
\end{array}\right).$$
In order to compute $X_{\Re a_s}$, we apply Lemma \ref{lemdiffX}. Note that 
$$(I+A^{-1}P X)^{-1}A^{-1}P=  \left(\begin{array}{cccc}\frac{2\varepsilon}{1+2\varepsilon\alpha}&0&0\\0&0&0\\0&0&\frac{\varepsilon}{1+\varepsilon\gamma}\end{array}\right)$$ 
and
$$-(I+A^{-1}P (X)^{-1}A^{-1} A_1(I-X)^2=\left(\begin{array}{cccc}-\frac{(1-\alpha)^2}{1+2\varepsilon\alpha}&0&0\\0&-1&0\\0&0&
-\frac{(1-\gamma)^2}{1+\varepsilon\gamma}\end{array}\right)$$ 
$$-(I+A^{-1}P (X)^{-1}A^{-1} A_2(I-X)^2= \left(\begin{array}{cccc}-\frac{(1-\alpha)^2}{1+2\varepsilon\alpha}&0&0\\0&1&0\\0&0&0\end{array}\right) $$
$$ -(I+A^{-1}P (X)^{-1}A^{-1} A_3(I-X)^2=\left(\begin{array}{cccc}0&0&0\\0&0&-(1-\gamma)^2\\0&\frac{1}{1+\varepsilon\gamma}&0\end{array}\right)$$
$$ -(I+A^{-1}P (X)^{-1}A^{-1} A_4(I-X)^2=\left(\begin{array}{cccc}0&\frac{-1}{1+2\varepsilon\alpha}&0\\-(1-\alpha)^2&0&0\\0&0&0\end{array}\right)$$
$$ -(I+A^{-1}P (X)^{-1}A^{-1} A_5(I-X)^2= \left(\begin{array}{cccc}0&0&\frac{-i(1-\gamma)^2}{1+2\varepsilon\alpha}\\0&0&0\\\frac{-i(1-\alpha)^2}{1+\varepsilon\gamma}&0&0\end{array}\right)$$
$$ -(I+A^{-1}P (X)^{-1}A^{-1} A_6(I-X)^2=\left(\begin{array}{cccc}0&0&0\\0&0&-i(1-\gamma)^2\\0&\frac{-i}{1+\varepsilon\gamma}&0\end{array}\right)$$
$$-(I+A^{-1}P (X)^{-1}A^{-1} A_7(I-X)^2=\left(\begin{array}{cccc}0&\frac{-i}{1+2\varepsilon\alpha}&0\\i(1-\alpha)^2&0&0\\0&0&0\end{array}\right).$$
We then find after a straightforward computation
$$ X_{\Re a_1} =   \left(\begin{array}{cccc}-\frac{(1-\alpha)^2}{1+4\varepsilon\alpha}&0&0\\0&-1&0\\0&0&
-\frac{(1-\gamma)^2}{1+2\varepsilon\gamma}\end{array}\right) \ X_{\Re a_2} =    \left(\begin{array}{cccc}-\frac{(1-\alpha)^2}{1+4\varepsilon\alpha}&0&0\\0&1&0\\0&0&0\end{array}\right) \ \  
X_{\Re a_3} =   \left(\begin{array}{cccc}0&0&0\\0&0&-(1-\gamma)^2\\0&\frac{1}{1+\varepsilon\gamma}&0\end{array}\right)\\
$$
$$ X_{\Re a_4}= \left(\begin{array}{cccc}0&\frac{-1}{1+2\varepsilon\alpha}&0\\-(1-\alpha)^2&0&0\\0&0&0\end{array}\right) \ \
 X_{\Re a_5}= \small \left(\begin{array}{cccc}0&0&-i\frac{(1-\gamma)^2}{1+2\varepsilon(\alpha+\gamma)}\\0&0&0
\\-i\frac{(1-\alpha)^2}{1+\varepsilon(\alpha+\gamma)}&0&0\end{array}\right)$$ 
$$X_{\Re a_6}= \left(\begin{array}{cccc}0&0&0\\0&0&-i(1-\gamma)^2\\0&\frac{-i}{1+\varepsilon\gamma}&0\end{array}\right) \ \ 
X_{\Re a_7}= \left(\begin{array}{cccc}0&\frac{-i}{1+2\varepsilon\alpha}&0\\i(1-\alpha)^2&0&0\\0&0&0\end{array}\right)$$
which leads to 
$$\SMALL (X_{\Re a_s}X^rV)_s=\left(\begin{array}{ccccccc}
-\frac{(1-\alpha)^2}{1+4\varepsilon\alpha}\alpha^r & -\frac{(1-\alpha)^2}{1+4\varepsilon\alpha}\alpha^r &0 & \frac{-\delta_{0r}}{1+2\varepsilon\alpha}& -i\frac{(1-\gamma)^2}{1+2\varepsilon(\alpha+\gamma)} \gamma^r& 0& \frac{-i\delta_{0r}}{1+2\varepsilon\alpha}\\
-\delta_{0r}&\delta_{0r} &-(1-\gamma)^2\gamma^r    & -(1-\alpha)^2\alpha^r&0& -i(1-\gamma)^2\gamma^r& i(1-\alpha)^2\alpha^r\\
-\frac{(1-\gamma)^2}{1+2\varepsilon\gamma}\gamma^r&0 &\frac{\delta_{0r}}{1+\varepsilon\gamma}     &0 & -i\frac{(1-\alpha)^2}{1+2\varepsilon(\alpha+\gamma)}\alpha^r& \frac{-i\delta_{0r}}{1+\varepsilon\gamma}& 0\\
\end{array}\right).$$  
Thus the matrix  $\Re \left(\sum_{r=0}^\infty \transp\overline{V}\left(I-\transp \overline {X}\right)^2 {\transp \overline {X}}^rK_j X_{\Re a_s}X^rV\right)_{j,s}$ is equal to
$$\Re \left(\sum_{r=0}^\infty \transp\overline{V}\left(I-\transp \overline {X}\right)^2 {\transp \overline {X}}^rK_j X_{\Re a_s}X^rV\right)_{j,s}= \left(\begin{array}{cc}C_1&(0)\\(0)&C_2\end{array}\right)$$
with 
$$ C_1= \left(\begin{array}{ccccccc}
-1-\frac{(1-\alpha)^2}{(1+4\varepsilon\alpha)(1+\alpha)^2} +\frac{(1-\gamma)^2}{(1+\gamma)^2(1+2\varepsilon\gamma)}& 1-\frac{(1-\alpha)^2}{(1+4\varepsilon\alpha)(1+\alpha)^2} & \
-\frac{2(1-\gamma)}{1+\gamma}  & \frac{-2(1-\alpha)}{1+\alpha}\\
1-\frac{(1-\alpha)^2}{(1+4\varepsilon\alpha)(1+\alpha)^2}&-1-\frac{(1-\alpha)^2}{(1+4\varepsilon\alpha)(1+\alpha)^2}&(1-\gamma)^2  &
 \frac{4\varepsilon\alpha(1-\alpha)}{(1+\alpha)(1+2\varepsilon\alpha)}\\
-2(1-\gamma)^2\frac{1+\varepsilon\gamma}{1+2\varepsilon\gamma}&(1-\gamma)^2&\frac{1}{1+\varepsilon\gamma}-\frac{(1-\gamma)^3}{1+\gamma}     &-\frac{(1-\alpha)^2(1-\gamma)^2}{1-\alpha\gamma} \\ 
-2(1-\alpha)^2\frac{1+2\varepsilon\alpha}{1+4\varepsilon\alpha}&4(1-\alpha)^2\frac{\varepsilon\alpha}{1+4\varepsilon\alpha}&-\frac{(1-\alpha)^2(1-\gamma)^2}{1-\alpha\gamma}     &-\frac{1}{1+2\varepsilon\alpha}-\frac{(1-\alpha)^3}{1+\alpha}\\
\end{array}\right)$$
and 
$$ C_2= \left(\begin{array}{ccccccc}
 \frac{1}{(1-\alpha\gamma)(1+2\varepsilon(\alpha+\gamma))}\left(\frac{(1-\alpha)^3}{1+\alpha}-\frac{(1-\gamma)^3}{1+\gamma}\right)&\frac{(1-\alpha)^2}{(1+\varepsilon\gamma)(1-\alpha\gamma)} &- \frac{(1-\gamma)^2}{(1-\alpha\gamma)(1+2\varepsilon\alpha)}\\
 \frac{(1-\alpha)^2}{1+2\varepsilon(\alpha+\gamma)}&\frac{1}{1+\varepsilon\gamma}-\frac{(1-\gamma)^3}{1+\gamma} & \frac{(1-\alpha)^2(1-\gamma)^2}{1-\alpha\gamma}\\
 -\frac{(1-\gamma)^2}{1+2\varepsilon(\alpha+\gamma)}&\frac{(1-\alpha)^2(1-\gamma)^2}{1-\alpha\gamma}& -\frac{1}{1+2\varepsilon\alpha}-\frac{(1-\alpha)^3}{1+\alpha}\\
\end{array}\right).$$  
Note that the invertibility of $C_2$ follows from the one of  
 $$ \left(\begin{array}{cccc}
\left(\frac{(1-\alpha)^3}{1+\alpha}-\frac{(1-\gamma)^3}{1+\gamma}\right)&\frac{(1-\alpha)^2}{1+\varepsilon\gamma} &- \frac{(1-\gamma)^2}{1+2\varepsilon\alpha}\\
(1-\alpha)^2&\frac{1}{1+\varepsilon\gamma}-\frac{(1-\gamma)^3}{1+\gamma} & \frac{(1-\alpha)^2(1-\gamma)^2}{1-\alpha\gamma}\\
-(1-\gamma)^2&\frac{(1-\alpha)^2(1-\gamma)^2}{1-\alpha\gamma}& -\frac{1}{1+2\varepsilon\alpha}-\frac{(1-\alpha)^3}{1+\alpha}\\
\end{array}\right).$$  
While slightly  lengthier, the invertibility of $C_1$ and $C_2$ follows from the same argument that led to the invertibility of the matrices $B_1$ \eqref{eqB_1} and $B_2$ \eqref{eqB_2} for the same small enough $\varepsilon\neq 0$. Alternatively, it can be checked directly that these four matrices are invertible for e.g. $\varepsilon=1/5$.

\end{proof}

\noindent {\bf Acknowledgements} The first  author was  supported by the Center for Advanced Mathematical Sciences and by an URB grant from the American University of Beirut. The second author was supported by the Czech Science Foundation (GACR) grant GC22-15012J.

\end{document}